\documentclass[12pt,a4paper,english,reqno]{amsart}
\usepackage[english]{babel}
\usepackage[T1]{fontenc}
\usepackage[utf8]{inputenc}

\usepackage{amsmath}
\usepackage{amsfonts}
\usepackage{amssymb}
\usepackage{amsthm}
\usepackage{graphicx}
\usepackage{verbatim}
\usepackage[textsize=tiny]{todonotes}
\usepackage{enumerate}
\usepackage{mathtools}
\usepackage{calrsfs}
\usepackage{hyperref}

\mathtoolsset{showonlyrefs=true}
\newtheorem{theorem}{Theorem}[section]
\theoremstyle{plain}

\newtheorem{definition}[theorem]{Definition}

\newtheorem{lemma}[theorem]{Lemma}

\newtheorem{observation}[theorem]{Observation}
\theoremstyle{definition}

\numberwithin{equation}{section}
\allowdisplaybreaks[1]

\begin{document}
\bibliographystyle{plain}
\title[Radial multipliers on amalgamated free products]{Radial multipliers on amalgamated free products of $\mathrm{II}_1$-factors}
\author{Sören Möller}
\address{Department of Mathematics and Computer Science \\ University of Southern Denmark \\ Campusvej 55 \\ 5230 Odense M \\ Denmark}
\email{moeller@imada.sdu.dk}
\date{\today}

\begin{abstract}
Let $\mathcal{M}_i$ be a family of $\mathrm{II}_1$-factors, containing a common $\mathrm{II}_1$-subfactor $\mathcal{N}$, such that $[\mathcal{M}_i:\mathcal{N}] \in \mathbb{N}_0$ for all $i$. Furthermore, let $\phi \colon \mathbb{N}_0 \to \mathbb{C}$. We show that if a Hankel matrix related to $\phi$ is trace-class, then there exists a unique completely bounded map $M_\phi$ on the amalgamated free product of the $\mathcal{M}_i$ with amalgamation over $\mathcal{N}$, which acts as an radial multiplier. Hereby we extend a result of U. Haagerup and the author for radial multipliers on reduced free products of unital $C^*$- and von Neumann algebras.
\end{abstract}

\maketitle

\section{Introduction}
Let $\mathcal{C}$ denote the set of functions $\phi$ on the non-negative integers $\mathbb{N}_0$ for which the matrix 
\begin{align*}
h = (\phi(i+j) - \phi(i+j+1))_{i,j \ge 0}
\end{align*}
is of trace-class. Let $(\mathcal{M},\omega) = \bar{\ast}_{i \in I} (\mathcal{M}_i,\omega_i)$ be the w*-reduced free product of von Neumann algebras $(\mathcal{M}_i)_{i \in I}$ with respect to normal states $(\omega_i)_{i \in I}$ for which the GNS-representation $\pi_{\omega_i}$ is faithful, for all $i \in I$.
In \cite{HaagerupMoeller2012} U. Haagerup and the author proved that if $\phi \in \mathcal{C}$, then there is a unique linear completely bounded normal map $M_\phi \colon \mathcal{M} \to \mathcal{M}$ such that $M_\phi(1)= \phi(0) 1$ and 
\begin{align}
M_\phi(a_1 a_2 \dots a_n) = \phi(n) a_1 a_2 \dots a_n
\end{align}
whenever $a_j \in \mathring{\mathcal{M}}_{i_j}=\operatorname{ker}(\omega_{i_j})$ and $i_1 \neq i_2 \neq \dots \neq i_n$.
Moreover,\linebreak $\| M_\phi \|_{cb} \le \|\phi\|_\mathcal{C}$ where $\| \cdot \|_\mathcal{C}$ is the norm on $\mathcal{C}$ defined in \eqref{eq_C_norm_amalgamated} below.

This result generalized the corresponding result for reduced $C^*$-alge\-bras of discrete groups which was proved by J. Wysocza{\'n}ski in \cite{Wysoczanski1995}.

In this paper we generalize above result to the case of amalgamated free products of $\mathrm{II}_1$-factors where amalgamation occurs over a integer-index $\mathrm{II}_1$-subfactor, resulting in the following theorem.

\begin{theorem}
\label{thm_amalgamatedRadial_main_theorem}
Let $(\mathcal{M},\mathbb{E}_\mathcal{N}) = \bar{\ast}_{\mathcal{N},i \in I} (\mathcal{M}_i,\mathbb{E}^i_\mathcal{N})$ be the amalgamated free product of $\mathrm{II}_1$-factors $(\mathcal{M}_i)_{i \in I}$ over a $\mathrm{II}_1$-subfactor $\mathcal{N}$ for which $[\mathcal{M}_i:\mathcal{N}] \in \mathbb{N}_0$ for all $i \in I$. Here $\mathbb{E}^i_\mathcal{N} \colon \mathcal{M}_i \to \mathcal{N}$, respectively, $\mathbb{E}_\mathcal{N} \colon \mathcal{M} \to \mathcal{N}$ denote the canonical conditional expectations.

If $\phi \in \mathcal{C}$, then there is a unique linear completely bounded normal right-$\mathcal{N}$-module map $M_\phi \colon \mathcal{M} \to \mathcal{M}$ such that $M_\phi(1)= \phi(0) 1$ and
\begin{align}
M_\phi(b_0 a_1 b_1 a_2 b_2 \dots a_n b_n ) = \phi(n) b_0 a_1 b_1 a_2 b_2 \dots a_n b_n
\end{align}
whenever $a_j \in \mathring{\mathcal{M}}_{i_j}=\operatorname{ker}(\mathbb{E}^i_\mathcal{N})$, $i_1 \neq i_2 \neq \dots \neq i_n$ and $b_0, \dots, b_n \in \mathcal{N}$. Moreover, $\| M_\phi \|_{cb} \le \|\phi\|_\mathcal{C}$.
\end{theorem}

Recently this result has been extended by S. Deprez to arbitrary amalgamated free products of finite von Neumann algebras in \cite{Deprez2013p} using different methods related to the construction of relative Fock spaces. It is an open problem if a similar result holds for even more general amalgamated free products. {\`E}. Ricard and Q. Xu investigated in  \cite[Section 5]{RicardXu2006} the amalgamated setting and the weaker estimate 
\begin{align}
\|M_\phi\|_{cb} \le |\phi(0)| + \sum_{n=1}^\infty 4n|\phi(n)|
\end{align}
can be obtained from \cite[Corollary 3.3]{RicardXu2006}.

We start by investigating some preliminaries in Section \ref{sec_amalgamated_preliminaries} followed by constructing the main building blocks of the radial multiplier in Section \ref{sec_amalgamated_technical}. The main result is then proved in Section \ref{sec_amalgamated_main_proof}.

\section{Preliminaries}
\label{sec_amalgamated_preliminaries}
Assume that $(\mathcal{N},\tau)$ is a $\mathrm{II}_1$-factor equipped with its finite trace $\tau$. Let now $I$ be some index set and
let $\mathcal{M}_i$, $i \in I$ be a family of $\mathrm{II}_1$-factors containing $\mathcal{N}$ equipped with the canonical conditional expectations $\mathbb{E}_\mathcal{N}^i \colon \mathcal{M}_i \to \mathcal{N}$. We will assume that the Jones index $[\mathcal{M}_i:\mathcal{N}]$ is an integer for all $i \in I$. Denote by $\tau_i = \tau \circ \mathbb{E}_\mathcal{N}^i$ the state on $\mathcal{M}_i$ and by $H_i = L^2(\mathcal{M}_i,\tau_i)$ the Hilbert space associated to $\mathcal{M}_i$ with respect to $\tau_i$. Observe that we can see $H_i$ as an $\mathcal{N}$-bimodule. Set $\mathring{H}_i = L^2(\mathcal{N},\tau)^\perp$ and set $\mathring{\mathcal{M}}_i = \mathcal{M}_i \cap \ker(\mathbb{E}_\mathcal{N}^i)$. Moreover, denote by $e^i_\mathcal{N}$ the projection from $L^2(\mathcal{M}_i,\tau_i)$ onto $L^2(\mathcal{N},\tau)$.
 
Let $(\mathcal{M},\tau) = \bar{\ast}_{\mathcal{N}, i \in I} \mathcal{M}_i$ be the amalgamated free product with respect to $\mathbb{E}_\mathcal{N}^i$ and let $\mathbb{E}_{\mathcal{N}} \colon \mathcal{M} \to \mathcal{N}$ be the conditional expectation comparable with $\tau$.

We have the amalgamated free Fock space
\begin{align}
\mathcal{F}_\mathcal{N} = L^2(\mathcal{N}) \oplus \bigoplus_{n \ge 1, i_1 \neq \cdots \neq i_n} \mathring{H}_{i_1} \otimes_\mathcal{N} \cdots \otimes_\mathcal{N} \mathring{H}_{i_n}
\end{align}
where $\otimes_\mathcal{N}$ fulfills the relation $x \otimes_\mathcal{N} by = xb \otimes_\mathcal{N} y$ for all $x \in \mathring{H}_i, y \in \mathring{H}_j, i \neq j$ and $b \in \mathcal{N}$. Note that there are natural left and right actions of $\mathcal{N}$ on $\mathcal{F}_\mathcal{N}$.
We denote by $\mathbb{B}_\mathcal{N}(\mathcal{F}_\mathcal{N})$ the subspace of $\mathbb{B}(\mathcal{F}_\mathcal{N})$ of maps that commute with the right action of $\mathcal{N}$.

Moreover, we will by $\langle \cdot, \cdot \rangle_\mathcal{N}$ denote the $\mathcal{N}$-valued inner product on $H_i$ coming from $\mathbb{E}_\mathcal{N}^i$ and extended to $\mathcal{F}$.

As M. Pimsner an S. Popa proved in \cite{PimsnerPopa1986} there will exist Pimsner-Popa-bases for these inclusions.
\begin{theorem}[{\cite[Proposition 1.3]{PimsnerPopa1986}}]
Let $\mathcal{M}_i$ and $\mathcal{N}$ as above and assume $[\mathcal{M}_i:\mathcal{N}] = N_i \in \mathbb{N}_0$. Set $n_i = N_i - 1$.
There exists a set $\Gamma_i=\{e_0^i, \dots, e_{n_i}^i \} \subset \mathcal{M}_i$ satisfying the properties:
\begin{itemize}
\item $\mathbb{E}_\mathcal{N}^i(e_k^{i*}e_l^i) = 0$ for $l \neq k$
\item $\mathbb{E}_\mathcal{N}^i(e_k^{i*}e_k^i) = 1$ for $0 \le k \le n_i$
\item $1 = \sum_{j=0}^{n_i} e_j^i e_\mathcal{N}^i e_j^{i*}$ as an equality in $\mathbb{B}(L_2(\mathcal{M}_i))$
\item For all $x \in \mathcal{M}_i$ we have
\begin{align}
\label{eq_amalgamated_right_base_expansion}
x = \sum_{j=0}^{n_i} \mathbb{E}_\mathcal{N}^i( x e_j^{i} ) e_j^{i*}
\end{align}
\end{itemize} 
\end{theorem}

We can without loss of generality choose $e_0^i = 1_\mathcal{N}^i$ for all $i \in I$, indeed, in the basis construction of \cite{PimsnerPopa1986}, one could choose $e^i_\mathcal{N}$ as the first projection and then use well-known facts on projections in $\mathrm{II}_1$-factors to complete the basis.

Set $\mathring{\Gamma}_i = \Gamma_i \setminus \{e_0^i\}$ and set $\mathring{\Gamma} = \bigcup_{i \in I} \mathring{\Gamma}_i$.
Denote 
\begin{align}
\Lambda(k) & = \{ b_0 \chi_1 b_1 \otimes_\mathcal{N} \cdots \otimes_\mathcal{N} \chi_{k-1} b_{k-1} \otimes_\mathcal{N} \chi_k b_k : \\
& \qquad \chi_j \in \mathring{\Gamma}_{i_j}, i_1 \neq i_2 \neq \cdots \neq i_n, b_0, \dots, b_k \in \mathcal{N} \} \subset \mathcal{F}_\mathcal{N}
\end{align}
and note that the span of $\Lambda(k), k \ge 0$ is dense in $\mathcal{F}_\mathcal{N}$.

\begin{definition}
\label{def_amalgamatedRadial_L_R}
For $\gamma \in  \mathring{\Gamma}_i$ and $\chi \in \Lambda(k)$ with $\chi_1 \in \mathring{\Gamma}_j$ and $\chi_k \in \mathring{\Gamma}_h$ we define
\begin{align}
L_\gamma(\chi) &= \left\{ \begin{array}{ll} \gamma \otimes_\mathcal{N} \chi & \qquad \text{ if } i \neq j \\ 0 & \qquad \text{ otherwise } \end{array} \right. \\
L^*_\gamma(\chi) &=  \mathbb{E}_\mathcal{N}(\gamma^* b_0 \chi_1) \chi' \quad \text{ if } \chi = b_0 \chi_1 \otimes_\mathcal{N} \chi' \\
R_{\gamma^*}(\chi) &= \left\{ \begin{array}{ll} \chi \otimes_\mathcal{N} \gamma^* & \qquad \text{ if } i \neq h \\ 0 & \qquad \text{ otherwise } \end{array} \right. \\
R^*_{\gamma^*}(\chi) &= \chi' \mathbb{E}_\mathcal{N}(\chi_k b_k \gamma) \quad \text{ if } \chi = \chi' \otimes_\mathcal{N} \chi_k b_k. 
\end{align}
\end{definition}

Note that if $\gamma \in \mathring{\Gamma}_i$ then $L_\gamma$, $L_\gamma^*$, $R_\gamma$ and $R_\gamma^*$ naturally sit in the copy of $\mathbb{B}_\mathcal{N}(H_i)$ inside $\mathbb{B}_\mathcal{N}(\mathcal{F}_\mathcal{N})$.

\begin{observation}
Observe that $L_\gamma$ and $L^*_\gamma$ commute with the right action of $\mathcal{N}$ on $\mathcal{F}_\mathcal{N}$, hence $L_\gamma, L^*_\gamma \in \mathbb{B}_\mathcal{N}(\mathcal{F}_\mathcal{N})$, while $R_{\gamma^*}$ and $R^*_{\gamma^*}$ commute with the left action of $\mathcal{N}$ on $\mathcal{F}_\mathcal{N}$.
\end{observation}

We also have to define sets of operators, which in the amalgamated setting replace the idea of $\mathbb{B}(H)$ in the non-amalgamated case.
\begin{definition}
\label{def_amalgamatedRadial_mathcalL}
Denote by $\mathcal{L}$ the set
\begin{align}
\mathcal{L} & = \{ b_0 L_{\xi_1} b_1 \cdots L_{\xi_k} b_k L^*_{\eta_1} \tilde{b}_1 \cdots L^*_{\eta_l} \tilde{b}_l : k,l \in \mathbb{N}_0, \xi_j, \eta_j \in \mathring{\Gamma}, b_i, \tilde{b}_j \in \mathcal{N} \}
\end{align}
and denote by $\tilde{\mathcal{L}}$ the weakly closed linear span of $\mathcal{L}$ in $\mathbb{B}_\mathcal{N}(\mathcal{F}_\mathcal{N})$.

Furthermore, we define
\begin{align}
\mathcal{L}^{k,l} & = \{ b_0 L_{\xi_1} b_1 \cdots L_{\xi_k} b_k L^*_{\eta_1} \tilde{b}_1 \cdots L^*_{\eta_l} \tilde{b}_l : \xi_j, \eta_j \in \mathring{\Gamma}, b_i, \tilde{b}_j \in \mathcal{N} \}.
\end{align}

Similarly, for $i \in I$ denote by 
\begin{align}
\mathcal{L}_i & = \{ b_0 L_{\xi_1} b_1 \cdots L_{\xi_k} b_k L^*_{\eta_1} \tilde{b}_1 \cdots L^*_{\eta_l} \tilde{b}_l : k,l \in \mathbb{N}_0, \xi_j, \eta_j \in \mathring{\Gamma}_i, b_i, \tilde{b}_j \in \mathcal{N} \}
\end{align}
and denote by $\tilde{\mathcal{L}_i}$ the weakly closed linear span of $\mathcal{L}_i$ in $\mathbb{B}_\mathcal{N}(H_i)$.

Furthermore, we define
\begin{align}
\mathcal{L}_i^{k,l} & = \{ b_0 L_{\xi_1} b_1 \cdots L_{\xi_k} b_k L^*_{\eta_1} \tilde{b}_1 \cdots L^*_{\eta_l} \tilde{b}_l : \xi_j, \eta_j \in \mathring{\Gamma}_i, b_i, \tilde{b}_j \in \mathcal{N} \}.
\end{align}
\end{definition}

\begin{observation}
\label{obs_amalgamatedRadial_short_L}
Note that by an argument similar to \cite[Lemma 3.1]{HaagerupMoeller2012} it would be enough to consider terms of the forms $b_0 L_{\gamma} b_1$, $b_0 L_{\gamma}^* b_1$ and $b_0 L_{\gamma} b_1 L_{\delta}^* b_2$ in the defintions of the spans of $\mathcal{L}_i$ and $\mathcal{L}_i^{k,l}$ as we can use equation \eqref{eq_amalgamated_right_base_expansion} to move the $b_i$ out of the way.
\end{observation}

\begin{definition}[{\cite[Definition 2.1]{HaagerupMoeller2012}}]
\label{def_amalgamatedRadial_class_C}
Let $\mathcal{C}$ denote the set of functions $\phi \colon \mathbb{N}_0 \to \mathbb{C}$ for which the Hankel matrix 
\begin{align}
h = (\phi(i+j)-\phi(i+j+1))_{i,j \ge 0}
\end{align}
is of trace-class.
\end{definition}

We will use the following facts about maps from $\mathcal{C}$ which are proven in \cite[Lemmas 4.6 and 4.8, Remark 4.7, after Definition 2.1]{HaagerupMoeller2012}. Here $\|x\|_1 = Tr(|x|)$ is the trace-class norm for $x \in \mathbb{B}(l^2(\mathbb{N}_0))$ and we use the notation $(u \odot v)(t) = \langle t,v \rangle u$, for $u,v,t \in l^2(\mathbb{N}_0)$.

\begin{lemma}
\label{lem_amalgamatedRadial_properties_psi1_psi2}
Let $\phi \in \mathcal{C}$. Then 
\begin{align}
k = (\phi(i+j+1)-\phi(i+j+2))_{i,j \ge 0}
\end{align}
is of trace-class. Furthermore, $c = \lim_{n \to \infty} \phi(n)$ exists and
\begin{align}
\sum_{n=0}^\infty \left| \phi(n)-\phi(n+1) \right| & \le \|h\|_1 + \|k\|_1 < \infty.
\end{align}

Let now 
\begin{align}
\label{eq_C_norm_amalgamated}
\| \phi \|_\mathcal{C} = \|h\|_1 + \|k\|_1 + |c|
\end{align}
and set for $n \ge 0$
\begin{align}
\psi_1(n) & = \sum_{i=0}^\infty (\phi(n+2i)-\phi(n+2i+1)) \\
\psi_2(n) & = \psi_1(n+1).
\end{align}
Then 
\begin{align}
\phi(n) = \psi_1(n) + \psi_2(n) + c
\end{align}
and for $i,j \ge 0$ the entries $h_{i,j}$ and $k_{i,j}$ of $h$ and $k$ are given by 
\begin{align}
h_{i,j} & = \psi_1(i+j) - \psi_1(i+j+2) \\
k_{i,j} & = \psi_2(i+j) - \psi_2(i+j+2).
\end{align} 

\label{remark_amalgamatedRadial_hk}
Furthermore, there exist $x_i, y_i, z_i, w_i \in l^2(\mathbb{N}_0)$ such that 
\begin{align}
h & = \sum_{i=1}^\infty x_i \odot y_i & \sum_{i=1}^\infty \|x_i\|_2 \|y_i\|_2 & = \|h\|_1 \\
k & = \sum_{i=1}^\infty z_i \odot w_i & \sum_{i=1}^\infty \|z_i\|_2 \|w_i\|_2 & = \|k\|_1.
\end{align}

\label{lem_amalgamatedRadial_psi_sum}
Moreover, we have
\begin{align}
\psi_1(k+l) & = \sum_{i=1}^\infty \sum_{t=0}^\infty x_i(k+t) \overline{y_i(l+t)} \\
\psi_2(k+l) & = \sum_{i=1}^\infty \sum_{t=0}^\infty z_i(k+t) \overline{w_i(l+t)}.
\end{align}
\end{lemma}

\section{Technical lemmas}
\label{sec_amalgamated_technical}

We will start by proving some technical lemmas, which will be useful in the main proof.

\begin{lemma}
\label{lem_amalgamated_e0_orthogonal}
Let $\gamma \in \mathring{\Gamma}_i$ and $b \in \mathcal{N}$ then
\begin{align}
\gamma b = \sum_{j=0}^{n_i} \mathbb{E}_\mathcal{N}^i( \gamma b e_j^{i} ) e_j^{i*} = \sum_{j=1}^{n_i} \mathbb{E}_\mathcal{N}^i( \gamma b e_j^{i} ) e_j^{i*}.
\end{align}
\end{lemma}

\begin{proof}
The first equality follows directly from equation \eqref{eq_amalgamated_right_base_expansion}. For the second equality observe that $\gamma \in \mathring{\Gamma}_i$ implies $\gamma = e_j^i$ for some $j \neq 0$ hence
\begin{align}
\mathbb{E}_\mathcal{N}^i( \gamma b e_0^{i} ) e_0^{i*} &= \mathbb{E}_\mathcal{N}^i( e_j^i b e_0^{i} ) e_0^{i*} \\
&=  \mathbb{E}_\mathcal{N}^i( e_j^i b 1 ) 1 \\
&= \mathbb{E}_\mathcal{N}^i( e_j^i ) b \\
&= 0 b = 0.
\end{align}
\end{proof}

Next we have to construct the building blocks for an explicit construction of the radial multiplier. Here we closely follow the outline of \cite[Section 4]{HaagerupMoeller2012} but generalize it to the amalgamated setting.

\begin{lemma}
\label{lem_amalgamatedRadial_rho_N_module_map}
The map $\rho \colon \tilde{\mathcal{L}} \to \mathbb{B}_\mathcal{N}(\mathcal{F}_\mathcal{N})$ defined by
\begin{align}
\rho(a) = \sum_{\gamma \in \mathring{\Gamma}} R_{\gamma^*} a R_{\gamma^*}^*
\end{align}
is a normal map that commutes with the right action of $\mathcal{N}$ on $\tilde{\mathcal{L}}$, respectively, $\mathbb{B}_\mathcal{N}(\mathcal{F}_\mathcal{N})$.
\end{lemma}

\begin{proof}
Let 
\begin{align}
a = b_0 L_{\xi_1} b_1 \cdots L_{\xi_k} b_k L^*_{\eta_1} \tilde{b}_1 \cdots L^*_{\eta_l} \tilde{b}_l \in \mathcal{L}^{k,l}
\end{align}
and 
\begin{align}
\chi = d_0 \chi_1 d_1 \otimes_\mathcal{N} \cdots \otimes_\mathcal{N} \chi_m d_m \in \Lambda(m)
\end{align}
then we have
\begin{align}
\rho(a)\chi &= \sum_{\gamma \in \mathring{\Gamma}} R_{\gamma^*} a R_{\gamma^*}^* \chi \\
&= \sum_{\gamma \in \mathring{\Gamma}} R_{\gamma^*} a \left(d_0 \chi_1 d_1 \otimes_\mathcal{N} \cdots \otimes_\mathcal{N} \chi_{m-1} d_{m-1} \mathbb{E}_\mathcal{N}(\chi_m d_m \gamma) \right).
\end{align}

Now
\begin{align}
& \sum_{\gamma \in \mathring{\Gamma}} R_{\gamma^*} a \left(d_0 \chi_1 d_1 \otimes_\mathcal{N} \cdots \otimes_\mathcal{N} \chi_{m-1} d_{m-1} \mathbb{E}_\mathcal{N}(\chi_m d_m \gamma) \right) \\
& = \sum_{\gamma \in \mathring{\Gamma}} R_{\gamma^*} b_0 L_{\xi_1} b_1 \cdots L_{\xi_k} b_k L^*_{\eta_1} \tilde{b}_1 \cdots L^*_{\eta_l} \tilde{b}_l \left(d_0 \chi_1 d_1 \otimes_\mathcal{N} \right. \\
& \qquad \qquad \left. \cdots \otimes_\mathcal{N} \chi_{m-1} d_{m-1} \mathbb{E}_\mathcal{N}(\chi_m d_m \gamma) \right) \\
& = \sum_{\gamma \in \mathring{\Gamma}} R_{\gamma^*} b_0 L_{\xi_1} b_1 \cdots L_{\xi_k} b_k c_l \left(d_l \chi_{l+1} d_{l+1} \otimes_\mathcal{N} \right. \\
& \qquad \qquad \left. \cdots \otimes_\mathcal{N} \chi_{m-1} d_{m-1} \mathbb{E}_\mathcal{N}(\chi_m d_m \gamma) \right) 
\end{align}
where $c_k \in \mathcal{N}$, $k=1 \dots l$, is iteratively defined by $c_0=1$, 
\begin{align}
c_k = \mathbb{E}_\mathcal{N}\left( \eta_{l-k+1} b_{l-k+1} c_{k-1} d_{k-1} \chi_k \right).
\end{align}

This gives us
\begin{align}
& \sum_{\gamma \in \mathring{\Gamma}} R_{\gamma} b_0 L_{\xi_1} b_1 \cdots L_{\xi_k} b_k c_l \left(d_l \chi_{l+1} d_{l+1} \otimes_\mathcal{N} \cdots \otimes_\mathcal{N} \chi_{m-1} d_{m-1} \mathbb{E}_\mathcal{N}(\chi_m d_m \gamma) \right)  \\
& = \sum_{\gamma \in \mathring{\Gamma}} R_{\gamma^*} \left( b_0 \xi_1 \otimes_\mathcal{N} \cdots \otimes_\mathcal{N} \xi_k b_k \otimes_\mathcal{N} c_l d_l \chi_{l+1} d_{l+1} \otimes_\mathcal{N} \right. \\
& \qquad \qquad \left. \cdots \otimes_\mathcal{N} \chi_{m-1} d_{m-1} \mathbb{E}_\mathcal{N}(\chi_m d_m \gamma) \right)
\end{align}
if $\xi_k$ and $\chi_{l+1}$ are from different $\mathring{\Gamma}_i$ as multiplying with an element from $\mathcal{N}$ does not change the Hilbert space and both side vanish if $\xi_k$ and $\chi_{l+1}$ are from the same $\mathring{\Gamma}_i$. Hence
\begin{align}
& \sum_{\gamma \in \mathring{\Gamma}} R_{\gamma^*} \left( b_0 \xi_1 \otimes_\mathcal{N} \cdots \otimes_\mathcal{N} \xi_k b_k \otimes_\mathcal{N} c_l d_l \chi_{l+1} d_{l+1} \otimes_\mathcal{N} \right. \\
& \qquad \qquad \left. \cdots \otimes_\mathcal{N} \chi_{m-1} d_{m-1} \mathbb{E}_\mathcal{N}(\chi_m d_m \gamma) \right) \\
& = \sum_{\gamma \in \mathring{\Gamma}} b_0 \xi_1 \otimes_\mathcal{N} \cdots \otimes_\mathcal{N} \xi_k b_k \otimes_\mathcal{N} c_l d_l \chi_{l+1} d_{l+1} \otimes_\mathcal{N} \\
& \qquad \qquad \cdots \otimes_\mathcal{N} \chi_{m-1} d_{m-1} \mathbb{E}_\mathcal{N}(\chi_m d_m \gamma) \otimes_\mathcal{N} \gamma^* \\
& = b_0 \xi_1 \otimes_\mathcal{N} \cdots \otimes_\mathcal{N} \xi_k b_k \otimes_\mathcal{N} c_l d_l \chi_{l+1} d_{l+1} \otimes_\mathcal{N} \\
& \qquad \qquad \cdots \otimes_\mathcal{N} \chi_{m-1} d_{m-1} \otimes_\mathcal{N} \sum_{\gamma \in \mathring{\Gamma}} \mathbb{E}_\mathcal{N}(\chi_m d_m \gamma) \gamma^* \\
& = b_0 \xi_1 \otimes_\mathcal{N} \cdots \otimes_\mathcal{N} \xi_k b_k \otimes_\mathcal{N} c_l d_l \chi_{l+1} d_{l+1} \otimes_\mathcal{N} \\*
& \qquad \qquad \cdots \otimes_\mathcal{N} \chi_{m-1} d_{m-1} \otimes_\mathcal{N} \chi_m d_m
\end{align}
in the case of $l<m-1$, which does not vanish as $\chi_{m-1}$ and $\chi_m$ are from different $\mathring{\Gamma}_i$ from the start. If instead $l = m-1$ we have
\begin{align}
& \sum_{\gamma \in \mathring{\Gamma}} R_{\gamma} \left( b_0 \xi_1 \otimes_\mathcal{N} \cdots \otimes_\mathcal{N} \xi_k b_k c_l d_l \mathbb{E}_\mathcal{N}(\chi_m d_m \gamma) \right) \\
& = \sum_{\gamma \in \mathring{\Gamma}} b_0 \xi_1 \otimes_\mathcal{N} \cdots \otimes_\mathcal{N} \xi_k b_k c_l d_l \mathbb{E}_\mathcal{N}(\chi_m d_m \gamma) \otimes_\mathcal{N} \gamma^* \\
& = b_0 \xi_1 \otimes_\mathcal{N} \cdots \otimes_\mathcal{N} \xi_k b_k c_l d_l \otimes_\mathcal{N} \sum_{\gamma \in \mathring{\Gamma}} \mathbb{E}_\mathcal{N}(\chi_m d_m \gamma) \gamma^* \\
& = b_0 \xi_1 \otimes_\mathcal{N} \cdots \otimes_\mathcal{N} \xi_k b_k c_l d_l \otimes_\mathcal{N} \chi_m d_m
\end{align}
if $\xi_k$ and $\chi_m$ are from different $\mathring{\Gamma}_i$ and vanishing otherwise.

Hence, in total we have
\begin{align}
\rho(a)\chi & = b_0 \xi_1 \otimes_\mathcal{N} \cdots \otimes_\mathcal{N} \xi_k b_k \otimes_\mathcal{N} c_l d_l \chi_{l+1} d_{l+1} \otimes_\mathcal{N} \\
& \qquad \qquad \cdots \otimes_\mathcal{N} \chi_{m-1} d_{m-1} \otimes_\mathcal{N} \chi_m d_m \\
& = a \chi
\end{align}
if $l<m-1$ or $l=m-1$ and $\xi_k$ and $\chi_m$ are from different $\mathring{\Gamma}_i$.

Now let $b \in \mathcal{N}$ then 
\begin{align}
\rho(a)(\chi b) &= a(\chi b) = a(\chi) b = (\rho(a) \chi)b
\end{align}
as multiplication from the right by $b$ neither changes $m$ and $l$ nor changes if $\xi_k$ and $\chi_m$ are from different $\mathring{\Gamma}_i$.

Normality follows from weak-$*$-continuity of multiplication and addition and hereby we can extend $\rho$ to all of $\tilde{\mathcal{L}}$ while preserving commutation with the right action of $\mathcal{N}$.
\end{proof}

As the next building block we construct the map $D_x$, closely following \cite[Section 4]{HaagerupMoeller2012} and omitting the proof as it is a simple calculation.

\begin{lemma}
\label{lem_amalgamatedRadial_D_N_module_map}
Let $x \in  l^\infty(\mathbb{N})$. The map $D_x \colon \mathcal{F}_\mathcal{N} \to \mathcal{F}_\mathcal{N}$ defined by
\begin{align}
D_x(\chi) = x_k \chi
\end{align}
for $\chi \in \Lambda(k)$ and extended by linearity is a $\mathcal{N}$-bimodule map and has the adjoint $D_{\bar{x}}$.
\end{lemma}

%\begin{proof}
%For the bimodule property observe that if $\chi \in \Lambda(k)$ and $b \in \mathcal{N}$ then $\chi b, b \chi \in \Lambda(k)$ as well.
	
%For the adjoint observe for $\chi \in \Lambda(k)$, $\xi \in \Lambda(l)$ that
%\begin{align}
%\mathbb{E}_\mathcal{N}(D_x(\chi)^*\xi) & = \mathbb{E}_\mathcal{N}((x_k \chi)^*\xi) \\
%& = \mathbb{E}_\mathcal{N}(\bar{x}_k \chi^* \xi) \\
%& = \mathbb{E}_\mathcal{N}(\chi^* \bar{x}_k \xi) \\
%& = \mathbb{E_\mathcal{N}}(\chi^* D_{\bar{x}} \xi).
%\end{align}
%Here the last equality holds as both sides vanish if $k \neq l$. Finally, we extend by right linearity over $\mathcal{N}$ to all of $\mathcal{F}_\mathcal{N}$. 
%\end{proof}

The last necessary building block is the map $\epsilon$ which, following \cite[Section 4]{HaagerupMoeller2012}, is constructed in the next lemma.

\begin{lemma}
\label{lem_amalgamatedRadial_epsilon_N_module_map}
Let $q_i$ be the projection on the closed linear span in $\mathcal{F}_\mathcal{N}$ of the set
\begin{align}
\left\{ \chi \in \Lambda(k) : k \ge 1, \chi = d_0 \chi_1 d_1 \cdots \chi_k d_k, d_0, \dots, d_k \in \mathcal{N}, \chi_k \in \mathring{H}_i \right\}
\end{align}
and define $\epsilon \colon \tilde{\mathcal{L}} \to \mathbb{B}_\mathcal{N}(\mathcal{F}_\mathcal{N})$ by $\epsilon(a) = \sum_{i \in I} q_i a q_i$. Then $\epsilon(a)$ is a normal map that commutes with the right action of $\mathcal{N}$ on $\tilde{\mathcal{L}}$, respectively, $\mathbb{B}_\mathcal{N}(\mathcal{F}_\mathcal{N})$.
\end{lemma}

\begin{proof}
To prove commutation with the right action of $\mathcal{N}$ observe that right multiplication of $\chi$ with $b \in \mathcal{N}$ will not change which $\mathring{H}_i$ $\chi_k d_k$ belongs to hence
\begin{align}
q_i (\chi b) = (q_i \chi) b
\end{align}
so $q_i \in \mathbb{B}_\mathcal{N}(\mathcal{F}_\mathcal{N})$ which implies $\epsilon(a) \in \mathbb{B}_\mathcal{N}(\mathcal{F}_\mathcal{N})$.

Normality follows from weak-$*$-continuity of multiplication and addition.
\end{proof}

The following lemma is modelled after \cite[Theorem 1.3]{ChristensenSinclair1989}, which states the similar result for the Hilbert space case, from which this follows as a simple observation.

\begin{lemma}
\label{lem_amalgamatedRadial_map_cb}
Let $\phi \colon \tilde{\mathcal{L}} \to \mathbb{B}_\mathcal{N}(\mathcal{F}_\mathcal{N})$ be given by
\begin{align}
\phi(a) = \sum_k u_k a v_k
\end{align}
for some bounded right-$\mathcal{N}$-module maps $u_k, v_k$ on $\mathcal{F}_\mathcal{N}$.
Then we have
\begin{align}
\| \phi \|_{cb} \le \| \sum_k u_k u_k^* \|^{1/2} \| \sum_k v_k^* v_k \|^{1/2}.
\end{align}
\end{lemma}

Combining the four preceeding lemmas we now can define maps that will be used to construct the radial multiplier explicitely and calculate bounds on its completely bounded norms.

\begin{lemma}
\label{lem_amalgamatedRadial_Phi1_Phi2_module_maps}
\label{lem_amalgamatedRadial_Phi1_norm}

For $x,y \in l^2(\mathbb{N}_0)$ and $a \in \tilde{\mathcal{L}}$ set 
\begin{align}
\Phi^{(1)}_{x,y} (a) & = \sum_{n=0}^\infty D_{(S^*)^n x} a D^*_{(S^*)^n y} + \sum_{n=1}^\infty D_{S^n x} \rho^n(a) D^*_{S^n y},
\end{align}
respectively,
\begin{align}
\Phi^{(2)}_{x,y} (a) & = \sum_{n=0}^\infty D_{(S^*)^n x} a D^*_{(S^*)^n y} + \sum_{n=1}^\infty D_{S^n x} \rho^{n-1}(\epsilon(a)) D^*_{S^n y}.
\end{align}

Then $\Phi^{(1)}_{x,y}, \Phi^{(2)}_{x,y} \colon \tilde{\mathcal{L}} \to \mathbb{B}_\mathcal{N}(\mathcal{F}_\mathcal{N})$ are well-defined, normal, completely bounded maps that commute with the right action of $\mathcal{N}$ on $\tilde{\mathcal{L}}$, respectively, $\mathbb{B}_\mathcal{N}(\mathcal{F}_\mathcal{N})$. Furthermore,
\begin{align}
\| \Phi^{(i)}_{x,y} \|_{cb} \le \|x\|_2 \|y\|_2,
\end{align}
for $i=1,2$.
\end{lemma}

\begin{proof}
Firstly, we by an argument akin to \cite[Lemma 4.2]{HaagerupMoeller2012} can observe that 
\begin{align}
\label{eq_amalgamatedRadial_sumDD_summDQD}
\sum_{n=0}^\infty D_{(S^*)^nx} D^*_{(S^*)^nx} + \sum_{n=1}^\infty D_{S^nx} Q_n D^*_{S^nx} = \|x\|^2 1_{\mathbb{B}(\mathcal{F}_\mathcal{N})}.
\end{align}
Here $Q_n$ denotes the projection on the span of simple tensors of length greater or equal $n$ in $\mathcal{F}_\mathcal{N}$.

Similarly, if we restrict us to $a \in \mathcal{L}$ we can observe 
\begin{align}
\rho^n(a) &= \sum_{\zeta \in \Lambda(n)} R_\zeta a R^*_\zeta, & a \in \tilde{\mathcal{L}} \\
\rho^{n-1}(\epsilon(a)) &= \sum_{\zeta \in \Lambda(n-1)} R_\zeta q_i a q_i R_\zeta^*, & a \in \tilde{\mathcal{L}}
\end{align}
and
\begin{align}
\rho^n(1) &= Q_n \\
\rho^{n-1}(\epsilon(1)) &= \rho^{n-1}(Q_1) = Q_n.
\end{align}
which by normality can be extended to all of $\tilde{\mathcal{L}}$.

This, similarly to \cite[Lemma 4.3]{HaagerupMoeller2012} gives us 
\begin{align}
\sum_{n=0}^\infty D_{(S^*)^nx} D^*_{(S^*)^nx} + \sum_{n=1}^\infty \sum_{\zeta \in \Lambda(n)} D_{S^nx} R_\zeta R^*_\zeta D^*_{S^nx} = \|x\|^2 1_{\mathbb{B}(\mathcal{F}_\mathcal{N})}
\end{align}
and
\begin{align}
\sum_{n=0}^\infty D_{(S^*)^nx} D^*_{(S^*)^nx} + \sum_{n=1}^\infty \sum_{\zeta \in \Lambda(n-1)} \sum_{i \in I} D_{S^nx} R_\zeta q_i R^*_\zeta D^*_{S^nx} = \|x\|^2 1_{\mathbb{B}(\mathcal{F}_\mathcal{N})}
\end{align}
and hence we from Lemmas \ref{lem_amalgamatedRadial_rho_N_module_map}, \ref{lem_amalgamatedRadial_D_N_module_map} and \ref{lem_amalgamatedRadial_epsilon_N_module_map} get that $\Phi^{(i)}_{x,y}$ is a well-defined, normal, completely bounded right-$\mathcal{N}$-module map on $\tilde{\mathcal{L}}$ and from Lemma \ref{lem_amalgamatedRadial_map_cb} that
\begin{align}
\| \Phi^{(i)}_{x,y} \|_{cb} \le \|x\|_2 \|y\|_2
\end{align}
for $i=1,2$ and for all $x,y \in l^2(\mathbb{N}_0)$.
\end{proof}

\begin{definition}
If $a = b_0 L_{\xi_1} b_1 \cdots L_{\xi_k} b_k L^*_{\eta_1} \tilde{b}_1 \cdots L^*_{\eta_l} \tilde{b}_l \in \mathcal{L}^{k,l}$ we say that we are in
\begin{itemize}
\item \textbf{Case 1} if $k=0$ or $l=0$ or $k,l \ge 1$ and $\xi_k \in \mathring{\Gamma}_i, \eta_l \in \mathring{\Gamma}_j$ and $i \neq j$, $i,j \in I$,
\end{itemize}
respectively,
\begin{itemize}
\item \textbf{Case 2} if $k,l \ge 1$ and $\xi_k, \eta_l \in \mathring{\Gamma}_i$ for some $i \in I$.
\end{itemize}
\end{definition}

\begin{lemma}
\label{lem_amalgamatedRadial_rhoLL_epsLL}
We have for all $n \ge 0$ and $a \in \mathcal{L}^{k,l}$ that $\rho^n( a ) = a Q_{l+n}$ and $\epsilon( a ) = \rho( a )$ in Case 1, and, respectively, $\epsilon( a ) = a$ in Case 2.

%Furthermore, we have the commutation relations $R_{\gamma_1} L_{\gamma_2} = L_{\gamma_2} R_{\gamma_1}$ for $\gamma_1, \gamma_2 \in \mathring{\Gamma}$ and 
%\begin{align}
%Q_{n+k} b_0 L_{\xi_1} b_1 \cdots L_{\xi_k} b_k & = b_0 L_{\xi_1} b_1 \cdots L_{\xi_k} b_k Q_n \\
%Q_{n+k} b_0 R_{\xi_1} b_1 \cdots R_{\xi_k} b_k & = b_0 R_{\xi_1} b_1 \cdots R_{\xi_k} b_k Q_n,
%\end{align}
%as well as 
%\begin{align}
%\sum_{\zeta_j, j=1,\dots, n} R_{\zeta_1} \cdots R_{\zeta_n} R_{\zeta_1}^* \cdots R_{\zeta_n}^* = Q_n
%\end{align}
%where the sum runs over all $\zeta_1, \dots, \zeta_n$ with $\zeta_j \in \mathring{\Gamma}_{i_j}$ and $i_1 \neq i_2 \neq \cdots \neq i_n$.
\end{lemma}

\begin{proof}
Following the argument of \cite[Lemma 4.4]{HaagerupMoeller2012} and using that $\xi b \in \mathring{H}_i$ if $\xi \in \mathring{H}_i$ and $b \in \mathcal{N}$. %The commutation relations can be easily verified.
\end{proof}

\begin{lemma}
\label{lem_amalgamatedRadial_Phi1_LL_eq_sum_abLL}
Let $k,l \ge 0$ and $a \in \mathcal{L}^{k,l}$. Then
\begin{align}
\Phi^{(1)}_{x,y} (a) = \left( \sum_{t=0}^\infty x(k+t) \overline{y(l+t)} \right) a
\end{align}
and, respectively, 
\begin{align}
\Phi^{(2)}_{x,y} (a) = \left\{ \begin{array}{ll}
\sum_{t=0}^\infty x(k+t) \overline{y(l+t)} a & \qquad \text{in Case 1} \\
\sum_{t=0}^\infty x(k+t-1) \overline{y(l+t-1)} a & \qquad \text{in Case 2}.
\end{array}
\right.
\end{align}
\end{lemma}

\begin{proof}
We prove this by showing that both sides act similarly on all simple tensors in $\mathcal{F}_\mathcal{N}$.
Following the argument of \cite[Lemma 4.5]{HaagerupMoeller2012} and using that $\xi b \in \mathring{H}_i$ if $\xi \in \mathring{H}_i$ and $b \in \mathcal{N}$ as well as the fact that the behavior of $D$ only depends on the length of simple tensors in $\mathcal{F}_\mathcal{N}$ but not on the elements from $\mathcal{N}$ contained in them.
\end{proof}

\section{Proof of the main result}
\label{sec_amalgamated_main_proof}
Now we have constructed all the necessary tools to prove the main result of this paper.

\begin{lemma}
\label{lem_amalgamatedRadial_equiv_T_mult_T_cases}
Let $T \colon \tilde{\mathcal{L}} \to \tilde{\mathcal{L}}$ be a bounded, linear, normal right-$\mathcal{N}$-module map, and let $\phi \colon \mathbb{N}_0 \to \mathbb{C}$. The following statements are equivalent.
\begin{enumerate}[(a)]
\item For all $n \ge 1$, $i_1, \dots i_n \in I$ with $i_1 \neq i_2 \neq \dots \neq i_n$ and $a_j \in \tilde{\mathcal{L}}_{i_j}$, we have $T(1) = \phi(0) 1$ and $T(a_1 a_2 \dotsm a_n) = \phi(n) a_1 a_2 \dotsm a_n$.
\item For all $k,l \ge 0$ and $a \in \mathcal{L}^{k,l}$ we have
\begin{align}
T( a ) = \left\{ \begin{array}{ll}
\phi(k+l) a & \qquad \text{in Case 1} \\
\phi(k+l-1) a & \qquad \text{in Case 2}.
\end{array} \right.
\end{align}
\end{enumerate}
\end{lemma}

\begin{proof}
(a) implies (b) follows by the same argument as in the proof of \cite[Lemma 5.2]{HaagerupMoeller2012} as $\mathcal{L} \subset \tilde{\mathcal{L}}$.

To prove (b) implies (a) let $a \in \tilde{\mathcal{L}}$ which by definition and right-$\mathcal{N}$-module property is in the weak closure of linear combinations of $\mathcal{L}$. Hence it by normality is enough to check that $T(1) = \phi(0) 1$ and $T(a_1 \dotsm a_n) = \phi(n) a_1 \dotsm a_n$ whenever $n \ge 1$ and $a = a_1 \cdots a_n \in \mathcal{L}$.

Hence $a = a_1 \dotsm a_n = b_0 L_{\gamma_1} b_1 \cdots L_{\gamma_k} b_k L^*_{\delta_1} \tilde{b}_1 \cdots L^*_{\delta_l} \tilde{b}_l$ for some $\gamma_j \in \mathring{\Gamma}_{i_j},$\linebreak $\delta_s \in \mathring{\Gamma}_{r_s}$, $i_j \neq i_{j+1}, r_s \neq r_{s+1}$ and $i_1, \dots, i_k, r_1, \dots, r_l \in I$.

If we are in Case 1, we have $i_k \neq r_l$. Hence neighboring elements on the right hand side are from different $\mathcal{L}_i$ and thus $n=k+l$. If we are in Case 2, we have $i_k = r_l$. Hence $L_{\gamma_k} b_k L^*_{\delta_l} \in \mathcal{L}_{i_k}$, thus $n=k+l-1$. Now (b) gives the result for $k \ge 1$ or $l \ge 1$.

Moreover, the $k=l=0$ case of (b) gives $T(b)=\phi(0)b$ for $b \in \mathcal{N}$ and hence $T(1)=\phi(0) 1$.
\end{proof}

Next, we explicitly construct such a map $T$.
\begin{definition}
Let $\phi \in \mathcal{C}$. Define maps
\begin{align}
T_1 = \sum_{i=1}^\infty \Phi^{(1)}_{x_i,y_i} 
\quad \text{and} \quad
T_2 = \sum_{i=1}^\infty \Phi^{(2)}_{z_i,w_i}
\end{align}
where $\Phi^{(\cdot)}_{x,y}$ are as in Lemma \ref{lem_amalgamatedRadial_Phi1_Phi2_module_maps}, $\psi_1, \psi_2$ as in Lemma \ref{lem_amalgamatedRadial_properties_psi1_psi2}, and $x_i, y_i, z_i, w_i$ as in Remark \ref{remark_amalgamatedRadial_hk}.
Moreover, define  $T = T_1 + T_2 + c \operatorname{Id}$ where $\operatorname{Id}$ denotes the identity operator on $\tilde{\mathcal{L}}$, and $c = \lim_{n \to \infty} \phi(n)$.
\end{definition}

First we prove that these maps are well-defined, normal, and completely bounded, afterwards we will prove that the maps exhibit the right behavior.

\begin{lemma}
\label{lem_amalgamatedRadial_T_cb_norm}
The maps $T_1, T_2$ and $T$ are normal and completely bounded right-$\mathcal{N}$-module maps and $\| T \|_{cb} \le \|\phi\|_{\mathcal{C}}$.
\end{lemma}

\begin{proof}
This lemma follows directly from Lemmas \ref{remark_amalgamatedRadial_hk} and \ref{lem_amalgamatedRadial_Phi1_norm}.
\end{proof}

\begin{lemma}
\label{lem_amalgamatedRadial_T1_T2_psi1_psi2}
\label{lem_amalgamatedRadial_T_LL_T_phi}
For $T_1, T_2$ defined as above, $k,l \ge 0$ and $a \in \mathcal{L}^{k,l}$ we have $T_1( a ) = \psi_1(k+l) a$,
respectively,
\begin{align}
\label{eq_amalgamatedRadial_T2_psi2}
T_2(a) = \left\{ \begin{array}{ll}
\psi_2(k+l) a & \qquad \text{in Case 1} \\
\psi_2(k+l-2) a & \qquad \text{in Case 2}.
\end{array} \right.
\end{align}

Furthermore, for $T$ defined as above we have
\begin{align}
T( a ) = \left\{ \begin{array}{ll}
\phi(k+l) a & \qquad \text{in Case 1} \\
\phi(k+l-1) a & \qquad \text{in Case 2.}
\end{array} \right.
\end{align}
\end{lemma}

\begin{proof}
Following the argument of \cite[Lemma 5.5]{HaagerupMoeller2012}.
\end{proof}

Note that by Lemma \ref{lem_amalgamatedRadial_equiv_T_mult_T_cases} this implies that $T(1) = \phi(0) 1$ and that for $n \ge 1$,  $T_\phi(a_1 a_2 \dotsm a_n) = \phi(n) a_1 a_2 \dotsm a_n$.

Finitely we have to check that combining Lemmas \ref{lem_amalgamatedRadial_T_cb_norm} and \ref{lem_amalgamatedRadial_T_LL_T_phi} is enough to prove Theorem \ref{thm_amalgamatedRadial_main_theorem}.

\begin{lemma}
\label{lem_amalgamatedRadial_M_in_L}
The von Neumann algebra $\mathcal{M}_i$ is isomorphic to a subalgebra of $\tilde{\mathcal{L}}_i$.
\end{lemma}
\begin{proof}
We can interpret both $\mathcal{M}_i$ and $\tilde{\mathcal{L}}_i$ as subalgebras of $\mathbb{B}(H_i)$.
Let $a \in \mathcal{M}_i$. Now set
\begin{align}
\tilde{a} = \sum_{j,k=0}^{n_i} L_{e_j^{i}} \mathbb{E}_\mathcal{N}^i(e_j^{i*} a e_k^{i}) L_{e_k^{i}}^* \in \tilde{\mathcal{L}}_i.
\end{align}

As $\mathcal{M}_i$ acts on $H_i = L^2(\mathcal{M}_i,\tau_i)$ we want to check that $a$ and $\tilde{a}$ act as the same operator on $H_i$. As $H_i$ is spanned over $\mathcal{N}$ from the right by $e_0^{i}, \dots, e_{n_i}^{i}$ it is enough to check that
\begin{align}
\langle a e_l^{i}, e_m^{i} \rangle_\mathcal{N} = \langle \tilde{a} e_l^{i}, e_m^{i} \rangle_\mathcal{N}
\end{align}
for all $l,m=0,\dots,n_i$.

But we have
\begin{align}
\langle \tilde{a} e_l^{i}, e_m^{i} \rangle_\mathcal{N} &= \langle \sum_{i,j=2}^n L_{e_j^{i}} \mathbb{E}_\mathcal{N}^i(e_j^{i*} a e_k^{i}) L_{e_k^{i}}^* e_l^{i}, e_m^{i} \rangle_\mathcal{N} \\
&= \sum_{j,k=0}^{n_i} \langle L_{e_j^{i}} \mathbb{E}_\mathcal{N}^i(e_j^{i*} a e_k^{i}) L_{e_k^{i}}^* e_l^{i}, e_m^{i} \rangle_\mathcal{N} \\
&= \sum_{j,k=0}^{n_i} \mathbb{E}_\mathcal{N}^i(e_m^{i*} L_{e_j^{i}} \mathbb{E}_\mathcal{N}^i(e_j^{i*} a e_k^{i}) L_{e_k^{i}}^* e_l^{i}) \\
&= \sum_{j,k=0}^{n_i} \delta_{k,l} \mathbb{E}_\mathcal{N}^i(e_m^{i*} L_{e_j^{i}} \mathbb{E}_\mathcal{N}^i(e_j^{i*} a e_k^{i})) \\
&= \sum_{j,k=0}^{n_i} \delta_{k,l} \mathbb{E}_\mathcal{N}^i(e_m^{i*} L_{e_j^{i}} 1) \mathbb{E}_\mathcal{N}^i(e_j^{i*} a e_k^{i})\\
&= \sum_{j,k=0}^{n_i} \delta_{k,l} \mathbb{E}_\mathcal{N}^i(e_m^{i*} e_j^{i} ) \mathbb{E}_\mathcal{N}^i(e_j^{i*} a e_k^{i})\\
&= \sum_{j,k=0}^{n_i} \delta_{k,l} \delta_{m,j} \mathbb{E}_\mathcal{N}^i(e_j^{i*} a e_k^{i})\\
&= \mathbb{E}_\mathcal{N}^i(e_m^{i*} a e_l^{i})\\
&= \langle a e_l^{i}, e_m^{i} \rangle_\mathcal{N}
\end{align}
as desired.
\end{proof}

\begin{lemma}
\label{lem_amalgamatedRadial_L_amalgamated_free_product_Li}
$\tilde{\mathcal{L}} = \ast_\mathcal{N} \tilde{\mathcal{L}}_i$.
\end{lemma}
\begin{proof}
By definition of the algebraic amalgamated free product of $\mathcal{L}_i$ over $\mathcal{N}$, it contains of all words of the form
\begin{align}
b_0 L^{\dagger}_{\xi_1} b_1 L^{\dagger}_{\xi_2} \cdots b_{n-1} L^{\dagger}_{\xi_n} b_n
\end{align}
where $\dagger$ is either $\ast$ or nothing and $\xi_j \in \mathring{\Gamma}_{i_j}$ and $i_1 \neq i_2 \neq \cdots \neq i_n$. But as multiplication with elements from $\mathcal{N}$ does not change which $\mathring{H}_i$ an element belongs to, all terms containing a sub-term of the form $L^*_\gamma b L_\delta$ will vanish, hence it is enough to consider terms of the form
\begin{align}
b_0 L_{\xi_1} b_1 L_{\xi_2} \cdots b_{n-1} L_{\xi_n} b_k L^*_{\eta_1} \tilde{b}_{1} \cdots L^*_{\eta_l} \tilde{b}_l
\end{align}
for neighbouring $\xi_i, \eta_i$ from different $\mathring{\Gamma}_i$ (although $\xi_k$ and $\eta_1$ might be from the same $\mathring{H}_i$) and $b_i \in \mathcal{N}$.  Here we use Observation \ref{obs_amalgamatedRadial_short_L} to be able to restrict us to short terms from each $\mathcal{L}_i$.

And all these terms by definition lie in $\mathcal{L}$.
On the other hand the span over $\mathcal{N}$ of those terms clearly lies in the amalgamated free product and by definition is $w^*$-dense in $\tilde{\mathcal{L}}_i$, hence the two von Neumann algebras are isomorphic.
\end{proof}

The preceeding two lemmas now allow us to prove the main theorem.

\begin{proof}[Proof of Theorem \ref{thm_amalgamatedRadial_main_theorem}]
$(\mathcal{M}_i,\mathbb{E}_\mathcal{N}^i)$ can be realized as subalgebras of $\tilde{\mathcal{L}_i}$ by Lemma \ref{lem_amalgamatedRadial_M_in_L} and hence $(\mathcal{M},\mathbb{E}_\mathcal{N}^i)$ by Lemma \ref{lem_amalgamatedRadial_L_amalgamated_free_product_Li} can be realized as a subalgebra of $\tilde{\mathcal{L}}$. Then as condition (b) and hence (a) holds for all of $\tilde{\mathcal{L}}$ by Lemma \ref{lem_amalgamatedRadial_equiv_T_mult_T_cases} it also holds for elements in the subalgebra $\mathcal{M}$ as the length of an operator in the amalgamated free product is preserved when restricting to a subalgebra. Hence $T$ restricted to the subalgebra has the desired behaviour and $M_\phi =T|_\mathcal{M}$ can be obtained by restricting $T$ to $\mathcal{M}$ and we then have $\|M_\phi\|_{cb} = \| T|_\mathcal{M} \|_{cb} \le \| T \|_{cb} \le \|\phi\|_\mathcal{C}$.
\end{proof}

\bibliography{thesis}

\end{document}